\documentclass[12 pt]{article}
\usepackage[pdftex]{graphicx}
\usepackage{amsmath}
\usepackage{mathtools}
\usepackage{amsthm}
\usepackage{amsfonts}
\usepackage{amssymb}
\usepackage[margin=1.0in]{geometry}
\usepackage{tikz-cd}
\usetikzlibrary{cd}
\usepackage{enumitem}

\newtheorem{theorem}{Theorem}

\newtheorem{lemma}[theorem]{Lemma}
\newtheorem{prop}[theorem]{Proposition}

\theoremstyle{definition}
\newtheorem{defin}[theorem]{Definition}

\theoremstyle{remark}
\newtheorem*{rem}{Remark}

\newcommand{\fin}[1]{\mathrm{Fin}(#1)}

\renewcommand{\phi}{\varphi}

\begin{document}
	\title{A note on minimal models for pmp actions}
	\author{Andy Zucker}
	\date{July 2019}
	\maketitle
	
\begin{abstract}
	Given a countable group $G$, we say that a metrizable flow $Y$ is \emph{model-universal} if by considering the various invariant measures on $Y$, we can recover every free measure-preserving $G$-system up to isomorphism. Weiss in \cite{W} constructs a minimal model-universal flow. In this note, we provide a new, streamlined construction, allowing us to show that a minimal model-universal flow is far from unique.
	\let\thefootnote\relax\footnote{2010 Mathematics Subject Classification. Primary: 37B05; Secondary: 28D15.}
	\let\thefootnote\relax\footnote{The author was supported by NSF Grant no.\ DMS 1803489.}
\end{abstract}
	
In this paper, we consider actions of an infinite countable group $G$ on a standard Borel probability space $(X, \mu)$ by Borel, measure-preserving bijections. When an action $a\colon G\times X\to X$ is understood, we will suppress the action notation, and given $g\in G$ and $x\in X$ just write $gx$ or $g\cdot x$ for $a(g, x)$. We will refer to $(X, \mu)$ as a \emph{$G$-system}. A $G$-system is \emph{free} if for $\mu$-almost every $x\in X$, we have $G_x = \{1_G\}$, where $G_x := \{g\in G: gx = x\}$ is the \emph{stabilizer} of $x\in X$. By passing to a subset of measure $1$, we will often implicitly assume that every point in a free $G$-system has trivial stabilizer. If $(X, \mu)$ and $(Y, \nu)$ are two $G$-systems, we say that $(Y, \nu)$ is a \emph{factor} of $(X, \mu)$ if there is a Borel $X'\subseteq X$ with $\mu(X') = 1$ and a Borel $G$-equivariant map $f\colon X'\to Y$ with $\nu = f^*\mu$. If we can find $f$ as above that is also injective, then we call $(X, \mu)$ and $(Y, \nu)$ \emph{isomorphic $G$-systems}. 

A \emph{$G$-flow} is an action of $G$ by homeomorphisms on a compact Hausdorff space. We similarly suppress the action notation. Given a $G$-system $(X, \mu)$, a \emph{model} for $(X, \mu)$ is a compact metric $G$-flow $Y$ and an invariant Borel probability measure $\nu$ so that $(X, \mu)$ and $(Y, \nu)$ are isomorphic $G$-systems. We will be most interested in \emph{minimal} $G$-flows, those $G$-flows in which every orbit is dense. Notice that any minimal model of a free $G$-system must be \emph{essentially free}, where a $G$-flow $Y$ is essentially free if for each $g\in G\setminus \{1_G\}$, the set $\{y\in Y: gy = y\}$ is nowhere dense.

We say that a metrizable $G$-flow $Y$ is \emph{model-universal} if by considering the various invariant measures $\nu$ on $Y$, the $G$-systems $(Y, \nu)$ recover every (standard) free $G$-system up to isomorphism. In \cite{W}, Weiss constructs for every countable group $G$ a minimal model-universal flow. It is natural to ask in what sense a minimal model-universal flow must be unique. Here, we prove a strong negative result. Given a family $\{Y_i: i\in I\}$ of minimal $G$-flows, we say that $\{Y_i: i\in I\}$ is \emph{mutually disjoint} if the product $\prod_{i\in I} Y_i$ is minimal. In particular, this implies that the $Y_i$ are pairwise non-isomorphic $G$-flows.
\vspace{2 mm}

\begin{theorem}
	\label{IntroThm:MinModels}
	For any countable group $G$, there is a mutually disjoint family $\{Y_i: i< \mathfrak{c}\}$ of minimal model-universal flows.
\end{theorem} 
\vspace{2 mm}

Let us call a $G$-flow $Y$ \emph{weakly model-universal} if for every free $G$-system $(X, \mu)$, there is an invariant measure $\nu$ on $Y$ so that $(Y, \nu)$ is a factor of $(X, \mu)$. In \cite{W}, Weiss first constructs a minimal, essentially free, weakly model-universal flow, then proves that any flow with these properties admits an almost 1-1 extension which is model-universal. We instead build our model-universal flows in one step.

A recent result of Elek in~\cite{E} shows the existence of a \emph{free} minimal model-universal flow. Recall that a $G$-flow $Y$ is free when for any $y\in Y$ and any $g\in G\setminus \{1_G\}$, we have $gy\neq y$. In the last section of this paper, we show how one can deduce this result using rather soft arguments.

\begin{theorem}
	\label{IntroThm:Free}
	Let $Y$ be a minimal, model-universal, Cantor flow. Then there is an almost 1-1 extension $\pi\colon Z\to Y$ so that $Z$ is free, minimal, and model-universal.
\end{theorem}
\vspace{2 mm}

As almost 1-1 extensions always preserve minimality and disjointness, we can strengthen Theorem~\ref{IntroThm:MinModels} as follows.
\vspace{2 mm}

\begin{theorem}
	\label{IntroThm:FreeDisjoint}
	For any countable group $G$, there is a mutually disjoint family $\{Y_i: i< \mathfrak{c}\}$ of free, minimal, model-universal flows.
\end{theorem}
\vspace{2 mm}

I would like to thank Benjamin Weiss for many helpful comments on an earlier draft, as well as the anonymous referee for suggesting many improvements.

\section{Basic examples of model-universal flows}
\label{Sec:Basic}

We briefly collect a few simple examples which will be important in what follows. Let $K$ be a compact space. Then $K^G$ is a $G$-flow with the right shift action, where given $g, h\in G$ and $s\in K^G$, we have $g\cdot s(h) = s(hg)$. Mostly we take $K = 2^n$ or $2^\omega$. 
\vspace{2 mm}

\begin{prop}
	\label{Prop:ModelUnivEx}
	The flow $(2^\omega)^G$ is model-universal.
\end{prop}

\begin{proof}
Let $(X, \mu)$ be a free $G$-system, and fix $\phi\colon X\to 2^\omega$ a Borel bijection. Now define $\psi\colon X\to (2^\omega)^G$ via $\psi(x)(g) = \phi(g\cdot x)$. Then $\psi$ is injective, and $(X, \mu)\cong ((2^\omega)^G, \psi^*\mu)$.
\end{proof}
\vspace{2 mm}

A \emph{subshift} of $K^G$ is a closed, $G$-invariant subspace. The following family of subshifts of $2^G$ will be an important source of weakly model-universal flows. Let $Q\subseteq G$ be a finite symmetric set. We say that $S\subseteq G$ is \emph{$Q$-spaced} if whenever $g, h\in S$ with $g\neq h$, then $Qg\cap Qh = \emptyset$. We say that $S$ is \emph{$Q$-syndetic} if we have $\bigcup_{g\in Q} gS = \bigcup_{g\in S} Qg = G$. Notice that maximal $Q$-spaced sets exist and are $Q^2$-syndetic. Conversely, any $Q^2$-syndetic $Q$-spaced set is a maximal $Q$-spaced set. We define
$$Y_Q = \{s\in 2^G: s^{-1}(\{1\})\text{ is a maximal $Q$-spaced set}\}.$$

\begin{prop}
	\label{Prop:SpacedWeakUniv}
	The flow $Y_Q$ is weakly model-universal.
\end{prop}

\begin{rem}
	This proposition is also one of the key ingredients used by Weiss (see~\cite{W}, Lemma 2.2).
\end{rem}

\begin{proof}
	Let $(X, \mu)$ be a free $G$-system. By freeness, we can find for every Borel $B\subseteq X$ with $\mu(B) > 0$ a Borel subset $A\subseteq B$ with $\mu(A) > 0$ and with $gA\cap A = \emptyset$ for any $g\in Q^2$. Let us call a Borel set $A$ with this property a \emph{$Q^2$-disjoint} set. Now if $\bigcup_{g\in Q^2} gA$ doesn't have full measure, we can find a $Q^2$-disjoint Borel set $A'\subseteq X$ with $\mu(A') > 0$ and $gA\cap A' = \emptyset$ for every $g\in Q^2$. As $Q$ is assumed symmetric, it follows that $A\cup A'$ is also $Q^2$-disjoint.
	
	Thus using a measure exhaustion argument, we can find $A\subseteq X$ a $Q^2$-disjoint Borel set so that $\mu\left(\bigcup_{g\in Q^2} gA\right) = 1$. We now let $\phi\colon X\to 2^G$ be the map given by $\phi(x)(g) = 1$ iff $gx\in A$. Then for almost every $x\in X$, $\phi(x)^{-1}(\{1\})$ is both $Q^2$-syndetic and $Q$-spaced, so a maximal $Q$-spaced set. It follows that $Y_Q$ contains the closed support of $\phi^*\mu$, so $(Y_Q, \phi^*\mu)$ is a factor of $(X, \mu)$.
\end{proof}
\vspace{2 mm}

We end the section by noting a simple closure property of (weakly) model-universal flows.
\vspace{0 mm}

\begin{prop}
	\label{Prop:Products}
	Let $Y_n$ be weakly model-universal $G$-flows. Then $Y := \prod_n Y_n$ is weakly model-universal. If at least one of the $Y_n$ is model-universal, then so is $Y$.
\end{prop}

\begin{proof}
	Let $(X, \mu)$ be a free $G$-system, and for each $n< \omega$, let  $\phi_n\colon X_n\to Y_n$ be a Borel, $G$-equivariant map, where $X_n\subseteq X$ satisfies $\mu(X_n) = 1$. Set $X' = \bigcap_n X_n$. Then $\mu(X') = 1$, and the map $\phi\colon X'\to \prod_n Y_n$ given by $\phi(x) = (\phi_n(x))_{n< \omega}$ is Borel and $G$-equivariant. If for some $n< \omega$, the map $\phi_n$ is injective, then $\phi$ will also be injective.
\end{proof}
\vspace{2 mm}

\section{Strongly irreducible subshifts}
\label{Sec:StrIrred}

The key technical tool we use here is the notion of a \emph{strongly irreducible subshift}. First, we introduce some general terminology.
Write $\fin{G}$ for the collection of finite subsets of $G$. Given $S_1, S_2\subseteq G$ and  symmetric $D\in \fin{G}$ with $1_G\in D$, we say that $S_1$ and $S_2$ are \emph{$D$-apart} if $DS_1\cap DS_2 = \emptyset$. Let $A$ be a finite set. If $Y\subseteq A^G$ is a subshift and $F\in \fin{G}$, we define the \emph{$F$-patterns} of $Y$ to be the set $S_F(Y) := \{s|_F: s\in Y\}\subseteq A^F$. Given $\alpha\in S_F(Y)$, we define the basic clopen neighborhood $N_Y(\alpha) := \{y\in Y: y|_F = \alpha\}$. If $F\in \fin{G}$, $S\subseteq G$, $\alpha\in A^F$, and $\beta\in A^S$, we say that $\alpha$ \emph{appears in} $\beta$ if there is $g\in G$ with $Fg\subseteq S$ and $\beta(fg) = \alpha(f)$ for each $f\in F$. We say in this case that $\alpha$ \emph{appears at} $g\in G$.

We say that $Y$ is \emph{strongly irreducible} if there is $D\in \fin{G}$ so that for any $F_0, F_1\in \fin{G}$ which are $D$-apart and any $\alpha_i\in S_{F_i}(Y)$, there is $y\in Y$ with $y|_{F_i} = \alpha_i$. We sometimes say that $Y$ is $D$-irreducible. We will frequently use the following facts about strongly irreducible subshifts. Here $A$ and $B$ are finite sets.
\begin{enumerate}
	\item 
	If $Y\subseteq A^G$ is $D_Y$-irreducible and $Z\subseteq B^G$ is $D_Z$-irreducible, then $Y\times Z\subseteq (A\times B)^G$ is $(D_Y\cup D_Z)$-irreducible.
	\item 
	Suppose $Y\subseteq A^G$ is $D$-irreducible and $\phi\colon Y\to B^G$ is continuous and $G$-equivariant. By continuity, there is $F\in \fin{G}$ so that $\phi(y)(1_G)$ depends only on $y|_F$. Then $Z := \phi[Y]$ is $DF$-irreducible.
\end{enumerate}
We will also need a method of making explicit choices of patterns in $S_F(Y)$. To that end, suppose that $A$ is linearly ordered, and enumerate the group $G$ in some fashion. This allows us to order $S_F(Y)$ lexicographically. We will use this ordering in the following two ways. Fix $Y\subseteq A^G$ a $D$-irreducible subshift.
\begin{enumerate}
	\item 
	If $F_0,...,F_{n-1}\in \fin{G}$ are pairwise $D$-apart, $\alpha_i\in S_{F_i}(Y)$, and $E\in \fin{G}$ contains each $F_i$, then we let $\mathrm{Conf}_Y(\alpha_0,...,\alpha_{n-1}, E)\in S_E(Y)$ be the lexicographically least $E$-pattern $\beta$ satisfying $\beta|_{F_i} = \alpha_i$.
	
	\item 
	Every strongly irreducible subshift is topologically transitive. In particular, fix $F\in \fin{G}$. Then for any $E\in \fin{G}$ containing at least $|S_F(Y)|$ many disjoint right translates of $DF$, there is $\beta\in S_E(Y)$ so that every $\alpha\in S_F(Y)$ appears in $\beta$. We let $\mathrm{Trans}_Y(F, E)$ be the lexicographically least $E$-pattern with this property.
\end{enumerate}
Most of the time, we take $A = 2^n$ for some $n< \omega$, and we take the lexicographic ordering on $2^n$ as the ordering on $A$.
\vspace{2 mm}

\section{The operator $\Phi$}
\label{Sec:Phi}

A subset $S\subseteq G$ is called \emph{syndetic} if $S$ is $Q$-syndetic for some $Q\in \fin{G}$. Given $F\in \fin{G}$ with $1_G\in F$ and $Y\subseteq A^G$ a subshift, we say that $Y$ is \emph{$F$-minimal} if for every $y\in Y$, every $\alpha\in S_F(Y)$ appears in $y$. Equivalently, for every $y\in Y$, every $\alpha\in S_F(Y)$ appears syndetically often. The following observation will be useful; suppose $Y\subseteq A^G$ is $F$-minimal and that every $\alpha\in S_F(Y)$ appears $E$-syndetically for some $E\in \fin{G}$. Then every such $\alpha$ appears in every $\beta \in S_{FE}(Y)$.

The following is our main method of producing strongly irreducible, $F$-minimal flows. First, recalling the flow $Y_Q$ from section~\ref{Sec:Basic}, we note that $Y_Q$ is $Q^3$-irreducible. Now let $Y\subseteq A^G$ be $D$-irreducible. Let $E\in \fin{G}$ be symmetric, contain $D$, and be large enough to contain at least $|S_F(Y)|\leq |A|^{|F|}$ many disjoint right translates of $DF$. Let $C\in \fin{G}$ be symmetric with $E^5\subseteq C$. We define a continuous, $G$-equivariant map $\phi\langle Y, F, E, C\rangle = \phi\colon Y\times Y_C\to A^G$ as follows. Suppose $(y, s)\in Y\times Y_C$, and write $z = \phi(y, s)$. Let $g\in G$.
\begin{itemize}
	\item 
	If $g = kh$, where $s(h) = 1$ and $k\in E$, set $z(g) = \mathrm{Trans}_Y(F, E)(k)$.
	\item 
	If there are not $k\in E^3$ and $h\in G$ with $s(h) = 1$ and $g = kh$, set $z(g) = y(g)$
	\item 
	If $g = kh$, where $s(h) = 1$ and $k\in E^3\setminus E$, set \\ 
	$z(g) = \mathrm{Conf}_Y(\mathrm{Trans}_Y(F, E),\, (h\cdot  y)|_{E^5\setminus E^3},\, E^5)(k)$.
\end{itemize}
The idea behind this definition is to reprint $y$ most of the time, using $s$ to tell us where to overwrite with the pattern $\mathrm{Trans}_Y(F, E)$, and using strong irreducibility to blend everything together. This construction is a slight modification of a construction in \cite{FTVF}; see their Figure 3 for a good illustration. 

It is routine to verify that $\phi$ as defined is continuous and $G$-equivariant. Denote by $\Phi(Y, F, E, C)$ the image of $\phi = \phi\langle Y, F, E, C\rangle$. Then $\Phi(Y, F, E, C)$ is $C^5$-irreducible. 
\vspace{2 mm}

\begin{lemma}
We have $S_F(Y) = S_F(\Phi(Y, F, E, C))$.
\end{lemma}
 
\begin{proof}
	The $\subseteq$ direction is clear. For the $\supseteq$ direction, suppose $z\in \Phi(Y, F, E, C)$ with $z = \phi(y, s)$. It is enough to show that $z|_F\in S_F(Y)$. If there is $h\in G$ with $s(h) = 1$ and $F\cap E^3h\neq \emptyset$, then $F\subseteq E^5h$, so we have
	$$z|_F = \mathrm{Conf}_Y(\mathrm{Trans}_Y(F,E), (h\cdot y)|_{E^5\setminus E^3}, E^5)|_F$$
	If there is no such $h\in G$, then we have $z|_F = y|_F$. 
\end{proof}
\vspace{2 mm}

For any $z\in \Phi(Y, F, E, C)$, the $E$-pattern $\mathrm{Trans}_Y(F, E)$ appears in $z$, so in particular every pattern in $S_F(Y)$ appears in $z$. Hence $\Phi(Y, F, E, C)$ is $F$-minimal.  Indeed, every $F$-pattern appears $C^3$-syndetically, since maximal $C$-spaced sets are $C^2$-syndetic. So every pattern in $S_F(Y)$ appears in every pattern in $S_{C^4}(\Phi(Y, F, E, C))$.

\vspace{2 mm} 

\section{A tree of subshifts}

We now use the operator $\Phi$ to produce a tree of strongly irreducible flows. We will construct for each $s\in 2^{<\omega}$ a strongly irreducible flow $X_s\subseteq (2^{|s|})^G$ by induction. This tree will be controlled by rapidly increasing sequences $\{D_k: k< \omega\}$, $\{E_k: k< \omega\}$, and $\{F_k: k< \omega\}$ of finite symmetric subsets of $G$. We will continue to add assumptions about how rapid this needs to be, but for now, we assume that 
\begin{itemize}
	\item 
	$\bigcup_n D_n = \bigcup_n E_n = \bigcup_n F_n = G$.
	\item 
	$E_n$ contains at least $2^{|D_n|(n+1)}$-many pairwise disjoint translates of $D_n^2$
	\item 
	$F_n\supseteq E_n^5$
	\item 
	$D_{n+1}\supseteq F_n^5$
\end{itemize}
Let $X_\emptyset$ be the trivial flow. If $s\in 2^{<\omega}$ and $X_s$ is defined, and $t = s^\frown 0$, then we set $X_t = X_s\times 2^G$. Suppose we are given $k< \omega$,  $s\in 2^k$, and $t = s^\frown 1\in 2^{k+1}$. Then we set $X_t = \Phi(X_s\times 2^G, D_k, E_k, F_k)$.

In order to discuss the key properties of this construction, we think of $(2^n)^G$ as embedded into $(2^\omega)^G$ by adding zeros to the end. In this way, we can refer to the \emph{$(n\times F)$-patterns} of a subflow $Y\subseteq (2^N)^G\cong 2^{N\times G}$, the set $S_{n\times F}(Y) := \{y|_{n\times F}: y\in Y\}$, whenever $N\geq n$.

\begin{enumerate}
	\item 
	\label{Enum:StrIrred}
	Each $X_s$ is $D_{|s|}$-irreducible.
	\item
	\label{Enum:PatternStable} 
	For any $s\sqsubseteq t\in 2^{<\omega}$ with $|s| = n$, we have $S_{n\times D_n}(X_s) = S_{n\times D_n}(X_t)$.
	\item 
	\label{Enum:Min}
	Suppose $s\in 2^{<\omega}$ is such that $|s| > n$ and $s(n) = 1$. Then every pattern in $S_{(n+1)\times D_n}(X_s)$ appears in every pattern in $S_{(n+1)\times D_{n+1}}(X_s)$.
	\item 
	\label{Enum:Inj}
	Suppose $s\in 2^n$. Then $S_{(n+1)\times D_{n+1}}(X_{s^\frown 0})\neq S_{(n+1)\times D_{n+1}}(X_{s^\frown 1})$. This is because the conclusion of item 3 is true for $X_{s^\frown 1}$ and false for $X_{s^\frown 0} = X_s\times 2^G$.
\end{enumerate}

We can now consider taking limits along the branches. It follows from item~\ref{Enum:PatternStable} above that for any $\alpha\in 2^\omega$, the flow $X_\alpha\subseteq (2^\omega)^G$ is well defined. We can think of $X_\alpha$ as a point in the space $K((2^\omega)^G)$ of compact subsets of $(2^\omega)^G$. The subshifts form a closed subspace, and given subshifts $\{Z_n: n< \omega\}\subseteq K((2^\omega)^G)$ and $Z\in K((2^\omega)^G)$, we have $Z_n\to Z$ iff for each finite $F\subseteq G$ and $k< \omega$, we eventually have $S_{k\times F}(Z_n) = S_{k\times F}(Z)$.  With this topology, the map $\Theta\colon 2^\omega\to K((2^\omega)^G)$ given by $\Theta(\alpha) = X_\alpha$ is continuous. Item~\ref{Enum:Inj} shows that $\Theta$ is injective. Whenever $\alpha\in 2^\omega$ has $\alpha^{-1}(\{1\})$ infinite, then item~\ref{Enum:Min} implies that $X_\alpha$ is a minimal flow.
\vspace{2 mm}

\begin{prop}
	For any $\alpha\in 2^\omega$ with $\alpha^{-1}(\{0\})$ and $\alpha^{-1}(\{1\})$ infinite, the flow $X_\alpha$ is a minimal, model-universal flow.
\end{prop}

\begin{proof}
Having already discussed minimality, we focus on model-universality. Write $T = \alpha^{-1}(\{1\})$, and form the flow $Y_\alpha:= (2^G)^\omega\times \prod_{n\in T} Y_{F_n}$. Then $Y_\alpha$ is model-universal. We have a continuous $G$-map $\psi_\alpha\colon Y_\alpha\to \prod_n X_{\alpha|_n}$ given inductively as follows. First let $f_\omega\colon \omega\to (\omega\setminus T)$ and $f_T\colon T\to T$ be infinite-to-one surjections. Let $y\in Y_\alpha$, and write $y = \{(y_n)_{n<\omega}, (s_n)_{n\in T}\}$ with $y_n\in 2^G$ and $s_n\in Y_{F_n}$. Then we write $\psi_\alpha(y) = (\psi_\alpha(y)_n)_{n<\omega}$ with each $\psi_\alpha(y)_n\in X_{\alpha|_n}$. We let $\psi_\alpha(y)_0$ be the unique member of the trivial flow $X_\emptyset$. If $\psi_\alpha(y)_n$ has been defined and $n\not\in T$, then $\psi_\alpha(y)_{n+1} = (\psi_\alpha(y)_n, y_{f_\omega(n)})$. If $n\in T$, then $\psi_\alpha(y)_{n+1} = \phi_n((\psi_\alpha(y)_n, s_{f_T(n)}), s_n)$, where $\phi_n = \phi\langle X|_{\alpha|_n}\times 2^G, D_n, E_n, F_n\rangle$. 

Notice that if the sequence $(\psi_\alpha(y)_n)_{n<\omega}$ converges to some $x\in (2^\omega)^G$, then $x\in X_\alpha$. Let $Y_\alpha'\subseteq Y_\alpha$ be the subset of those $y$ for which $\psi_\alpha(y)_n$ is convergent. Then the map $\eta\colon Y'_\alpha \to X_\alpha$ with  $\eta(y) = \lim_n \psi_\alpha(y)_n$ is Borel. It suffices to show that if the $D_n$ grow rapidly enough, then $Y_\alpha'$ has measure $1$ for any $G$-invariant measure on $Y_\alpha$. To that end, fix $y = ((y_n)_{n<\omega}, (s_n)_{n\in T})$, and consider some $g\in G$. A sufficient condition for the sequence $\psi_\alpha(y)_n(g)$ to be convergent is that for a tail of $n\in T$, we have $s_n(h) = 0$ whenever $h\in E_n^3g$. This condition ensures that for suitably large $n\in T$, we have $\psi_\alpha(y)_{n+1}(g) = (\psi_\alpha(y)_n(g), s_{f_T(n)}(g))$. Define $Y''_\alpha\subseteq Y'_\alpha$ to be those $y$ for which on a tail of $n\in T$, we have $s_n(g) = 0$ for any $g\in E_n^4$. Notice that $Y''_\alpha$ is also Borel and $G$-invariant.

Fix $\nu$ an invariant measure on $Y_{F_n}$. Then letting $U = \{s\in Y_{F_n}: s(1_G) = 1\}$, we have $\nu(U) \leq 1/|F_n|$. This is because $g\cdot U = \{s\in Y_{F_n}: s(g^{-1}) = 1\}$, so by definition of the subshift $Y_{F_n}$, we have that the collection $\{g\cdot U: g\in F_n\}$ is pairwise disjoint. Then by invariance and a union bound, we have $\nu(\{s\in Y_{F_n}: s(g) = 1 \text{ for some }g\in E_n^4\}) \leq |E_n^4|/|F_n|$. We now add our last assumption to the growth of the $D_n$.
\begin{itemize}
	\item 
	$|E_n^4|/|F_n| < 1/2^n$.
\end{itemize}
From this assumption, it follows from the Borel-Cantelli lemma that for any invariant measure $\mu$ on $Y_\alpha$ that $\mu(Y_\alpha'') = 1$. 

Furthermore, we claim that $\eta$ is injective on $Y_\alpha''$. To see this, suppose that $y\neq y'\in Y_\alpha''$, with $y = \{(y_n)_{n<\omega}, (s_n)_{n\in T}\}$ and $y' = \{(y'_n)_{n< \omega}, (s'_n)_{n\in T}\}$. First suppose that $y_n(g)\neq y_n'(g)$ for some $n < \omega$ and $g\in G$. Then for some large enough $N < \omega$ and any $k, \ell\geq N$, we have $\psi_\alpha(y)_k(g) = \psi_\alpha(y)_\ell(g)$, and same for $y'$. Now pick some suitably large $k\in \omega\setminus T$ with $f_\omega(k) = n$. Then $\psi_\alpha(y)_{k+1}(g) = \psi_\alpha(y)_k(g)\times y_n(g)$, and similarly for $y'$. It follows that $\eta(y)\neq \eta(y')$. In the case that $s_n(g)\neq s_n'(g)$ for some $n\in T$, the argument is almost the same. For a suitably large $k\in T$ with $f_T(k) = n$, we use the assumption that $y$ and $y'$ are in $Y_\alpha''$ to see that $\psi_\alpha(y)_{k+1}(g) = \psi_\alpha(y)_k(g)\times s_n(g)$, and similarly for $y'$. Once more, we have $\eta(y)\neq \eta(y')$.
\end{proof}
\vspace{2 mm}

To prove Theorem~\ref{IntroThm:MinModels}, we need to recall some results from~\cite{GTWZ} (in particular, see Corollary 6.8). There, it is shown that every minimal flow is disjoint from every strongly irreducible subshift. From this, it follows that every minimal flow is disjoint from any $X_\alpha$ where $\alpha$ has a tail of zeros. Since disjointness is a $G_\delta$ condition (\cite{GTWZ}, Proposition 6.4), it follows that every minimal flow is disjoint from $X_\alpha$ for comeagerly many $\alpha\in 2^\omega$. We are now in a position to apply Mycielski's theorem (see \cite{K}, 19.1) to find our mutually disjoint family $\{X_{\alpha_i}: i< \mathfrak{c}\}$ of minimal, model-universal shifts.
\vspace{2 mm}

\section{From essentially free to free}

Recall that if $Y$ is a minimal metrizable flow, then an extension $\pi\colon Z\to Y$ is called \emph{almost 1-1} if the set $\{z\in Z: |\pi^{-1}(\{\pi(z)\})| = 1\}$ is comeager. Notice that $Z$ must also be minimal. To see this, let $z\in Z$ and $V\subseteq Z$ be non-empty open. Then find $z'\in V$ with $|\pi^{-1}(\{\pi(z')\})| = 1$. We can find a net $g_i\in G$ with $g_i\cdot \pi(z)\to \pi(z')$. It follows that $g_i\cdot z\to z'$. In particular, the orbit of $z$ meets $V$. 

One method of producing almost 1-1 extensions of a given minimal $G$-flow is to consider $\mathrm{Reg}(Y)$, the Boolean algebra of regular open subsets of $Y$. Recall that $A\subseteq Y$ is \emph{regular open} if $\mathrm{Int}(\overline{A}) = A$. We remind the reader that in this Boolean algebra, we have $A^c = Y\setminus \overline{A}$, $A\vee B = \mathrm{Int}(\overline{A\cup B})$, and $A\wedge B = A\cap B$. If $\mathcal{B}\subseteq \mathrm{Reg}(Y)$ is a subalgebra, then $\mathrm{St}(\mathcal{B})$, the space of ultrafilters on $\mathcal{B}$, is a compact, zero-dimensional space whose basic clopen neighborhood has the form $\{p\in \mathrm{St}(\mathcal{B}): A\in p\}$, where $A\in \mathcal{B}$. If $\mathcal{B}$ is also $G$-invariant, then $\mathrm{St}(\mathcal{B})$ is a $G$-flow. If $\mathcal{B}$ is countable, then $\mathrm{St}(\mathcal{B})$ is homeomorphic to Cantor space. Now suppose that $\mathcal{B}$ contains a basis for the topology on $Y$. Then we have a $G$-map $\pi\colon \mathrm{St}(\mathcal{B})\to Y$ given by $\pi(p) = y$ iff every $A\in \mathcal{B}$ with $A\ni y$ satisfies $A\in p$. Furthermore, the map $\pi$ is \emph{pseudo-open}, meaning that images of open sets have non-empty interior. For $y\in Y$, we have $|\pi^{-1}(\{y\})| = 1$ iff for every $A\in \mathcal{B}$, we have $y\in A$ or $y\in Y\setminus \overline{A}$. So when $\mathcal{B}$ is countable, the set $\{y\in Y: |\pi^{-1}(y)| = 1\}$ is comeager. Since $\pi$ is pseudo-open, it follows that $\{z\in Z: |\pi^{-1}(\pi(z))| = 1\}$ is also comeager.

In general, an almost 1-1 extension can have very different measure-theoretic behavior than the base flow. Indeed, this fact is heavily exploited in~\cite{W}. For us however, we will seek to build almost 1-1 extensions which preserve the measure-theoretic properties of the base flow. For the remainder of the section, fix $Y$ a minimal, model-universal flow whose underlying space is a Cantor set. Recall that this implies that $Y$ is essentially free. We will call an invariant measure $\mu$ on $Y$ \emph{free} if for every $g\in G$, we have $\mu(Y_g) = 0$, where $Y_g = \{y\in Y: gy = y\}$.
\vspace{2 mm}

\begin{defin}
	\label{Def:StrReg}
	Given $A\subseteq Y$, we call $A$ \emph{strongly regular open} if $A$ is regular open and for every free invariant measure $\mu$, we have $\mu(A) + \mu(Y\setminus \overline{A}) = 1$. Denote by $\mathrm{SReg}(Y)$ the collection of strongly regular open sets.
\end{defin} 
\vspace{2 mm}

\begin{prop}
	\label{Prop:StrBoolean}
	$\mathrm{SReg}(Y)$ is a $G$-invariant subalgebra of $\mathrm{Reg}(Y)$.
\end{prop}

\begin{proof}
	Clearly $\mathrm{SReg}(Y)$ is $G$-invariant and closed under complements, so it is enough to check closure under intersection. Given $A, B\in \mathrm{SReg}(Y)$, we have 
	\begin{align*}
	\overline{(A\cap B)}\setminus (A\cap B) &= \overline{(A\cap B)}\setminus A\cup \overline{(A\cap B)}\setminus B\\[1 mm]
	&\subseteq (\overline{A}\setminus A)\cup (\overline{B}\setminus B).
	\end{align*}
	Since $A$ and $B$ are both strongly regular open, the last entry must have measure zero for any free invariant measure $\mu$.
\end{proof}
\vspace{2 mm}

Of course, we have yet to prove the existence of any interesting strongly regular open sets. We do this in the next lemma.
\vspace{2 mm}

\begin{lemma}
	\label{Lem:StrRegExist}
	For every $g\in G\setminus \{1_G\}$, there is a partition of $Y\setminus Y_g$ into three relatively clopen pieces $A_g$, $B_g$, and $C_g$ with the property that $gA_g\cap A_g = \emptyset$, and likewise for $B_g$ and $C_g$. In particular, $A_g$, $B_g$, and $C_g$ are all strongly regular open sets.
\end{lemma}

\begin{proof}
	Write $Y\setminus Y_g = \bigcup_n U_n$ with each $U_n$ compact open. We may assume that the $U_n$ are pairwise disjoint, and by further partitioning each $U_n$ into finitely many clopen pieces if needed, we may assume that $gU_n\cap U_n = \emptyset$ for each $n< \omega$. We will inductively partition $V_n := \bigcup_{k< n} U_n$ into pieces $A_n$, $B_n$, and $C_n$ with the property that $A_N\cap V_n = A_n$ for $N\geq n$, likewise for $B_N$ and $C_N$. We then set $A_g = \bigcup_n A_n$, and likewise for $B_g$ and $C_g$. 
	
	We set $A_0 = B_0 = C_0 = \emptyset$. Assume $A_k$, $B_k$, and $C_k$ have been defined for some $k < \omega$. We will form clopen sets $A_k'$, $B_k'$, and $C_k'$ so that $U_k = A_k'\cup B_k'\cup C_k'$.  Partition $U_n$ into finitely many clopen sets $\{W_j: j<m\}$ with the property that for each $j< m$ and for each $h\in \{g^{-1}, g\}$, we either have $hW_j\subseteq A_k$, $hW_j\subseteq B_k$, $hW_j\subseteq C_k$,  or $hW_j\cap (A_k\cup B_k\cup C_k) = \emptyset$. Add each $W_j$ to the set $A_k'$, $B_k'$, or $C_k'$ in such a way so that if $hW_j\subseteq A_k$ for some $h$ as above, then $W_j$ is not added to $A_k'$, and likewise for $B_k'$ and $C_k'$. We then set $A_{k+1} = A_k\cup A_k'$, and likewise for $B_{k+1}$ and $C_{k+1}$.
	
	Notice that for each $n< \omega$, we have $gA_n\cap A_n = \emptyset$, and likewise for $B_n$ and $C_n$. Hence $A_g$ will also satisfy $gA_g\cap A_g = \emptyset$ as desired, and likewise for $B_g$ and $C_g$.
\end{proof}
\vspace{2 mm}

The last lemma we will need shows that metrizable, almost 1-1 extensions of $Y$ using strongly regular open sets preserve the measure-theoretic properties of $Y$.
\vspace{2 mm}

\begin{lemma}
	\label{Lem:MeasureInjPart}
	Let $\mathcal{B}$ be a countable $G$-invariant subalgebra of $\mathrm{SReg}(Y)$ extending the clopen algebra of $Y$. Let $Z = \mathrm{St}(\mathcal{B})$, and let $\pi\colon Z\to Y$ be the associated almost 1-1 extension. Then for any free invariant measure $\mu$ on $Y$, we have $\mu(\{y: |\pi^{-1}(\{y\})| = 1\}) = 1$.
\end{lemma}

\begin{proof}
	By the discussion at the beginning of the section, we have
	\begin{align*}
	\{y\in Y: |\pi^{-1}(\{y\})| = 1\} = \bigcap_{A\in \mathcal{B}} A\cup (Y\setminus \overline{A}).
	\end{align*}
	Since $\mathcal{B}$ is a countable collection of strongly regular open sets, this set must have measure $1$ for any free $\mu$.
\end{proof}
\vspace{2 mm}

\begin{proof}[Proof of Theorem~\ref{IntroThm:Free}]
	Let $\mathcal{B}\subseteq \mathrm{SReg}(Y)$ be a countable, $G$-invariant subalgebra containing all of the sets $A_g$, $B_g$, $C_g$ from Lemma~\ref{Lem:StrRegExist}. Then $\mathrm{St}(\mathcal{B})$ will be the desired flow. To see that $\mathrm{St}(\mathcal{B})$ is free, let $p\in \mathrm{St}(\mathcal{B})$ and $g\in G\setminus \{1_G\}$. Then $p$ contains one of $A_g$, $B_g$, or $C_g$, WLOG say $A_g\in p$. Then since $gA_g\cap A_g = \emptyset$, we must have $gp\neq p$. To see that $\mathrm{St}(\mathcal{B})$ is model-universal, we note that on the set $Y_0 := \{y\in Y: |\pi^{-1}(\{y\})| = 1\}$, the map $\pi^{-1}\colon Y_0\to Z$ is well defined. By Lemma~\ref{Lem:MeasureInjPart}, this set has measure $1$ for all free invariant measures on $Y$.
\end{proof}

\end{document}